\documentclass[12pt,a4paper]{article}

\usepackage[utf8]{inputenc}
\usepackage[T1]{fontenc}

\usepackage{amsmath}
\usepackage{amsthm}
\usepackage{amssymb}
\usepackage{amsfonts}
\usepackage{dsfont}
\usepackage{url}
\usepackage{graphicx}
\usepackage{hyperref}
\usepackage{color}
\usepackage{tikz}
\usepackage{pgfplots}
\usetikzlibrary{shapes.misc}
\usepackage{enumerate}
\usepackage{fullpage}

\theoremstyle{plain}
\newtheorem{theorem}{Theorem}[section]
\newtheorem{lemma}[theorem]{Lemma}

\newtheorem{remark}[theorem]{Remark}
\newtheorem{as}{Assumption}

\pgfplotsset{compat=1.13}
\renewcommand{\P}{\mathbb{P}} %probability measure
\newcommand{\E}{\mathbb{E}} %expectation
\newcommand{\N}{\mathbb{N}} %natural numbers
\newcommand{\R}{\mathbb{R}} % real numbers
\newcommand{\limn}{\lim_{n \to \infty}}

\newcommand{\Addresses}{{%
  \bigskip
  \footnotesize

  P.~Dyszewski \textsc{Instytut Matematyczny Uniwersytetu Wroc\l{}awskiego, Pl. Grunwaldzki 2/4 50-384, %newline
   Wroc\l{}aw, Poland}\par\nopagebreak
 \textit{E-mail address:} \texttt{piotr.dyszewski@math.uni.wroc.pl}\par\nopagebreak
and\par\nopagebreak
\textsc{Fakultät für Mathematik, Technische Universität München, \newline 
Boltzmannstr.~3, 85748 Garching, Germany}\par\nopagebreak
  \medskip

  N.~Gantert, \textsc{Fakultät für Mathematik, Technische Universität München, \newline 
  Boltzmannstr.~3, 85748 Garching, Germany}\par\nopagebreak
  \textit{E-mail address:} \texttt{gantert@ma.tum.de}
  
  \medskip
  
  T.~Höfelsauer, \textsc{Fakultät für Mathematik, Technische Universität München, \newline 
  Boltzmannstr.~3, 85748 Garching, Germany}\par\nopagebreak
  \textit{E-mail address:} \texttt{thomas.hoefelsauer@tum.de}

 }}

\begin{document}
\title{Large deviations for the maximum of a branching random walk with stretched exponential tails}
\author{Piotr Dyszewski, Nina Gantert and Thomas H\"ofelsauer}
\maketitle

\begin{abstract}
	We prove large deviation results for the position of the rightmost particle, denoted by $M_n$, in a one-dimensional branching random walk in a case when Cram\'er's condition is not 
	satisfied. More precisely we consider step size distributions with stretched exponential upper and lower tails, i.e.~both tails decay as $e^{-|t|^r}$ for some $r\in( 0,1)$. It is known that in this case, $M_n$ grows as 
	$n^{1/r}$ and in particular	faster than linearly in $n$. Our main result is a large deviation principle for the laws of 
	$n^{-1/r}M_n$ . In the proof we use a comparison with the maximum of (a random number of) independent random walks, denoted by $\tilde M_n$, and we show a large deviation principle for the laws of 
	$n^{-1/r}\tilde M_n$ as well.
 
\medskip
\noindent \textbf{Keywords:} branching random walk, large deviations, stretched exponential random variables

\smallskip
\noindent \textbf{AMS 2000 subject classification:} 60F10, 60J80, 60G50. 
\end{abstract}

\section{Introduction} \label{intro}

	We study branching random walk, which is a discrete time Galton-Watson processes with a spatial component. Given a reproduction law and a step size distribution the 
	evolution of the process can be described as follows.  At time $n=0$ we place one particle at the origin of the real line. At time $n=1$ this particle splits according to 
	the reproduction law and each new particle performs an independent step, according to the step size distribution. We assume that the branching mechanism and the displacements 
	are independent. From here, each particle evolves in the same way,
	independently of all other particles. \\
%We refer to Section~\ref{prel} for a more formal description of the model. \\ 
	We are interested in the position of the rightmost particle at time $n \in \mathbb{N}_0$, which we will denote by $M_n$. In the case when the step size distribution satisfies Cram\'er's condition,
	the asymptotic behaviour of $M_n$ is fairly well understood (see the recent monograph~\cite{S15} and references therein). We will investigate a case of heavy-tailed steps, 
	when Cram\'er's condition if not satisfied. More specifically, we consider the case of steps with stretched exponential distribution, when both tails decay as 
	$e^{-|t|^r}$ for some $r\in( 0,1)$. Then it is known that $M_n$ grows like $n^{1/r}$ as established by a law of large numbers proved in~\cite{G00}.
	In the present article we will provide the corresponding large deviation results. Our main result is Theorem \ref{theorem_BRW}  which gives a large deviation principle for the laws of $M_n/n^{1/r}$.
This complements previous results on large deviations for branching random walks given by
\cite{BM17}, \cite{CH17}, \cite{CH18}, \cite{GH18}, \cite{LP15}, \cite{LT17}, \cite{R93}.

The article is organized as follows. In Section~\ref{prel} we present the necessary preliminaries concerning the step size distribution and the branching mechanism and provide the necessary notation.
% followed by a precise description of our model. 
 The main results are presented in Section~\ref{sec:main} and are proved in Section~\ref{sec:proofs}.

\section{Preliminaries}\label{prel}

We write $f \sim g$ for two functions $f,g \colon \N \to \R$ whenever $f(n)/g(n) \to 1$ as $n \to \infty$. We call a sequence $(a_n)_{n \geq 0}$ subexponential if for 
	any $\varepsilon>0$, $a_n e^{-\varepsilon n} \to 0$ as $n \to \infty$. 
	
	\subsection{Step size distribution}
Let $X, X_1, X_2, \ldots$ be a collection of iid random variables of zero mean and let $S = (S_n)_{n \geq 0}$ be the corresponding random walk, that is $S_0=0$, 
	$S_n = \sum_{k=1}^n X_k$. 
	In the case when Cram\'er's condition holds, i.e.
	\begin{equation}\label{eq:cramer}
	\E[e^{sX}]<\infty \quad \mbox{for some }s>0 
	\end{equation}
	it is well known that
	$\lim_{n \to \infty} n^{-1} \log \P(S_n>xn) = -\sup_{s \geq 0}\left\{sx - \log \E[e^{sX}] \right\}$ for  $x >0$,	 see~\cite{DZ10}. If on the other hand, $\E[e^{sX}]=\infty$ for any $s>0$, it is known that the probabilities $\P(S_n>xn)$ decay slower than exponentially in $n$ with the exact rate
	 being determined by the behaviour of the tail $\P(X>x)$ as $x \to \infty$. We will focus on the case of stretched exponential distribution.

	\begin{as} \label{assumption_RW}
		The random variable $X$ is centred $(\E [ X ] =0)$ and has stretched exponential upper and lower tails, that is there exist $\lambda_+, \lambda_->0$, $r \in (0,1)$ and slowly varying 
		functions $a_+$ and $a_-$ such that
		\begin{equation}\label{upptail}
			\P(X \geq x)=a_+(x) e^{-\lambda_+ x^r}
			\end{equation}
			and
			\begin{equation}\label{lowtail}
			 \P(X \leq -x)=a_-(x) e^{-\lambda_- x^r}
		\end{equation}
		for all $x\geq 0$.
	\end{as}
\begin{remark}
We assume \eqref{lowtail} for simplicity, to avoid an annoying number of case 
distinctions: our strategy of proof can be used for other lower tails, see also Remark \ref{moregeneral}.
\end{remark}

	Large deviations for a random walk in the case when Cram\'er's condition is not fulfilled go back to~\cite{N69} where it was established that if the law of $X$ has a density that decays as $e^{-|x|^r}$ as $|x| \to \infty$, then
$\lim_{n \to \infty} n^{-r} \log \P(S_n > nx) = -x^r$
	for $x>0$. It is easy to see, that in order to obtain an exponential rate of decay under Assumption~\ref{assumption_RW} one has to consider probabilities of the 
	form $\P(S_n \geq xn^{1/r})$. 

	\begin{lemma} \label{lem:LDP_RW}
		Let Assumption \ref{assumption_RW} hold. Then, the laws of $\frac{S_n}{n^{1/r}}$ satisfy a large deviation principle with rate function $I$ given by
			\begin{equation} \label{def_I}
				I(x)= \begin{cases}
				\lambda_+ x^r & \text{for } x \geq 0\\
				\lambda_- (-x)^r & \text{for } x \leq 0.\\
				\end{cases}
		\end{equation}
	\end{lemma}

	This lemma follows, for example, from the main results in~\cite{GRR14}.

\subsection{Branching mechanism} 
	Let $Z = (Z_n)_{n \geq 0}$ be a Galton-Watson process with $Z_0 =1$ and reproduction law $ p =(p(k))_{k\geq 0}$. 
The key parameter describing the asymptotic behaviour of 
	$Z$ is the mean of $p$ denoted by
	\begin{equation*}
		m = \sum_{k \geq 0} k p(k).
	\end{equation*}
	It is well-known (assuming $p(1)<1$ to rule out a degenerate case)  that the branching process survives with positive probability if and only if $m >1$. More precisely, if we denote by $q= \P(\lim_{n \to \infty} Z_n =0)$ 
	the extinction probability, then $q$ is the smallest solution of the equation $f(s) = s$,
	where $f$ denotes the probability generating function of $p$, that is $f(s) = \sum_{k\geq 0}s^kp(k)$. If $m  = f'(1-)>1$, then we see that $q<1$. 
	On the event of survival, $Z_n \to \infty$  as $n \to \infty$ and the Kesten-Stigum Theorem (see Theorem~\ref{Lemma_GWP_martingale}) asserts that $Z_n/m^n$ converges almost surely to a finite and strictly positive random variable,
	provided that $\E[ Z_1 \log Z_1] <\infty$. 
	%We will prove our main result under the two aforementioned conditions.
Our assumptions on the branching process will be the following.
	\begin{as} \label{as:BP}
The Galton-Watson process $Z$ is supercritical, that is $m>1$, and $\E[ Z_1 \log Z_1] <\infty$.
	\end{as}
	
	We introduce the conditional probability
	\begin{equation}\label{Pstardef}
		\P^*(\: \cdot \: ) = \P(\: \cdot \:  | \: Z_n>0, \:\forall \: n \in\N ).
	\end{equation}
	We will need the rate of decay of the probabilities that $Z_n$ grows slower than $m^n$. It turns out 
	that one has to distinguish between two cases, namely $p(0) + p(1) > 0$ and 
	$p(0) +p(1) = 0$. The first one is often referred to as the Schröder case, while the latter is called Böttcher case. We see that in the second case ($p(0)+p(1)=0$), $Z$ should 
	grow faster since $Z_n \geq 2^n$ for all $n$. Denote the set of all possible values of $Z$ by
	\begin{equation} \label{def_A}
		A= \bigl\{l \in \N \colon \exists n \in \N \text{ such that } \P(Z_n=l)>0 \bigr\}.
	\end{equation}
	Let $k^*$ be smallest possible number of offspring, that is
	\begin{equation} \label{def_k*}
		k^*=\inf\{k \geq 1 \colon p(k)>0 \}.
	\end{equation}
	Note that $k=k^*$ is the smallest positive integer, such that $\P(Z_n=k)>0$ for some $n \in \N$. In the Böttcher case we have $k^*\geq 2$. 

\subsubsection{Schröder case: $p(0)+p(1)>0$}
	In this case the rate of deviations for the branching process $Z$ is determined by
	\begin{equation}\label{eq:defRho}
		\rho = - \log f'(q) = -\log \E \left[Z_1 q^{Z_1-1}\right] \in (0, \infty)\, .
	\end{equation}
	Note that $\rho = -\log p(1)$ if $p(0)=0$ (and therefore $q=0$). Moreover $p(0)+p(1)>0$ is necessary and sufficient for $\rho<\infty$.
	We summarize several results concerning the lower deviations of $Z$ in the following Lemma. 
	
	\begin{lemma}\label{lem:LDBP_S}
		Let Assumption~\ref{as:BP} be in force and suppose that $p(0)+p(1)>0$. Then,
		\begin{itemize}
			\item For any $k \in A$,
				\begin{equation*}
					\limn \frac{1}{n} \log \P^* \bigl(Z_n =k \bigr) = - \rho.
				\end{equation*}
			\item For every subexponential sequence $(a_n)_{n \geq 0}$ such that $a_n \to \infty$ as $n \to \infty$,
				\begin{equation*}
					\limn \frac{1}{n} \log \P^* \bigl(Z_n \leq a_n \bigr) = - \rho.
				\end{equation*}
			\item For any $x \in [0, \log m]$,
				\begin{equation}\label{def_I_GW_S}
					\limn \frac{1}{n} \log \P^* \bigl( Z_n \leq e^{xn} \bigr) = -  I^{\textnormal{GW}}(x) := -\rho \bigl(1-x (\log m)^{-1} \bigr).
				\end{equation} 
		\end{itemize}
	\end{lemma}
	
	The first claim of Lemma~\ref{lem:LDBP_S} can be found in  \cite[Chapter 1, Section 11, Theorem 3]{AN04}, the the last two are special cases of~ \cite[Theorem 4]{FW06}.
	 Note that in all three cases precise asymptotics for the deviations are known.
	In our arguments we will only make use of the logarithmic asymptotics. 

\subsubsection{Böttcher case: $p(0)+p(1)=0$}
	
	We see that whenever $p(0)+p(1)=0$, $\rho$ given in~\eqref{eq:defRho} is infinite. Thus one may expect that $ \P^* \bigl( Z_n \leq e^{xn} \bigr)$ decays faster than 
	exponentially in the Böttcher case.  For example, one can see directly that
	\begin{equation*}
		 \P\left(Z_n=(k^*)^n\right) = \exp\left \{ \frac{(k^*)^n-1}{k^*-1} \log p(k^*) \right\},
	\end{equation*}
	where $k^*$ is given in~\eqref{def_k*}. It turns out that $k^*$ is the key parameter describing the lower deviations of $Z$.  
	We will make use of a result established in~\cite{FW06}.
	
	\begin{lemma}\label{lem:LDBP_B}	
		For any $k_n = o(m^n)$ such that $k_n \geq (k^*)^n$, there are positive constants $B_1, B_2$ such that 
		\begin{align*}
			-B_1 & \leq \liminf_{n \to \infty} (k^*)^{b_n-n} \log \P(Z_n \leq k_n) \\
				& \leq \limsup_{n \to \infty} (k^*)^{b_n-n} \log \P(Z_n \leq k_n) \leq -B_2,
		\end{align*}
		where 
		\begin{equation*}
			b_n = \min \{ j \: | \: m^j(k^*)^{n-j} \geq 2 k_n \}.
		\end{equation*}
	\end{lemma}

\subsection{Branching random walk}
	
	We now define the branching random walk which is a discrete time stochastic process that evolves in the following way.
	At time $n=0$ one particle is placed at the origin of the real line. 
	This particle will start a population which will be described by the branching process $Z = (Z_n)_{n\geq 0}$. At time $n=1$ the initial particle splits into $Z_1$ new particles which 
	move independently of each other and $Z_1$. We assume that
	each displacement from the place of birth is an independent copy of $X$. Each particle evolves according to these rules independently of all others. More precisely at time $n=2$, 
	each particle, independently of the others, splits into a random number of particles distributed according to $p$. The total number of particles present in the system at time $n=2$ 
	is denoted by $Z_2$. Each particle performs, independently of other particles and $Z_1, Z_2$, a step distributed as $X$. The system continues to evolve according to these rules.  
	Let $\mathcal{T} = (V, E)$ be the associated Galton-Watson tree with the initial particle denoted as the root $o \in V$. 
	Let $D_n \subset V$ denote the 
	set of particles present at time $n$. Clearly $|D_n | = Z_n$. 
	For $v, w \in V$ write $[v,w]$ for the set of edges along the unique line from $v$ to $w$. To model the displacements, assume that each edge of the tree $\mathcal{T}$ is labeled 
	with an independent copy of $X$, that is we are given  
a collection	$\{ X_e\}_{e\in E}$ of iid random variables distributed as $X$. Then the position of the particle $v$ is equal to
	\begin{equation*}
		S_v = \sum_{e \in [o,v]} X_e
	\end{equation*} 
	and the position of the rightmost particle at time $n$ is
	\begin{equation*}
		M_n = \max_{v \in D_n}S_v.
	\end{equation*}
 	It is well known, that if \eqref{eq:cramer} is satisfied and the step size distribution has mean $0$, then $M_n$ has a linear speed, that is $n^{-1}M_n$ converges to a strictly positive constant a.s. (see \cite{B76, H74, K75}). 
	In this case large deviations for $M_n$ were already considered in a number of papers, see \cite{BM17, CH17, CH18, GH18, LP15, LT17, R93}. We refer to~\cite{S15} for a more general model and many limit laws and concentration properties.

	In our case, as verified in~\cite{G00}, under Assumptions~\ref{assumption_RW} and \ref{as:BP}, $M_n$ grows with superlinear speed. 
	\begin{lemma}
		Let Assumptions~\ref{assumption_RW} and \ref{as:BP} be in force. Then 
		\begin{equation}\label{linspeed}
			\limn \frac{M_n}{n^{1/r}} = \alpha:= \Bigl( \frac{\log m}{\lambda_+} \Bigr)^{1/r} \quad \P^*\text{-a.s.}
		\end{equation}
	\end{lemma}	
	
	A more detailed description of the asymptotic behaviour of $M_n$, giving also the second term, i.e.~the order of $M_n -n^{1/r}( \frac{\log m}{\lambda_+} )^{1/r}$, was recently delivered in~\cite{DGH20}. 
	Our main results stated in the next section concern large deviations related to~\eqref{linspeed}.  Clearly, $M_n$ is a maximum of random walks which are pairwise dependent. 
	In some aspects however, the asymptotic behaviour of $M_n$ is similar to that of a 
	maximum of (a random number of) independent random walks. We will use the latter as a reference model. Consider a collection $(X^{(j)}_k)_{j,k\geq 1}$ of iid random variables distributed as $X$ and put 
	$S_0^{(j)}=0$ and 
	\begin{equation}\label{defcopy}
	S_n^{(j)}=\sum_{k=1}^nX_k^{(j)}\, .
	\end{equation}
	 We then define the maximum of (a random number of) independent random walks by
	\begin{equation}\label{eq:Mind}
		\tilde{M}_n = \max_{1\leq j\leq Z_n}S_n^{(j)}.	
	\end{equation}
	Since $\tilde{M}_n$ involves independent random variables it stochastically dominates $M_n$ which involves dependent random walks, see Lemma \ref{lemma_BRW_vs_ind_RW2}. As a consequence, we can use the deviations of $\tilde{M}_n$ as a bound for the deviations of $M_n$. For completeness, we also establish large deviations for the laws of $n^{-1/r}\tilde{M}_n$. 

\section{Main results}\label{sec:main}
Recall \eqref{linspeed} and let, for $x < \alpha$,
\begin{equation} \label{def_H}
		H(x) = 
		\left\{ \begin{array}{cc} k^*\lambda_- (\alpha-x)^r & \mbox{if } p(0)+p(1)=0 \\
		  \min\left\{ \rho + \lambda_-(-x)^r, \lambda_- (\alpha-x)^r \right\}  & \mbox{if } x \leq 0,\: p(0)+p(1)>0\\
		   \min \left\{ \lambda_- \alpha^r,  \rho  \right\}\cdot \left(1 - \left( \frac x\alpha\right)^r \right) & \mbox{if } x \geq 0, \:p(0)+p(1)>0.
				\end{array} \right.
	\end{equation}

	Define the rate function of the branching random walk as
	\begin{equation} \label{def_I_BRW}
		I^{\text{BRW}}(x)=
		\begin{cases}
		\lambda_+ x^r - \log m & \text{for } x \geq \alpha,\\
		H(x) & \text{for } x < \alpha.
		\end{cases}
	\end{equation}
Our main result is the following.
	\begin{theorem}\label{theorem_BRW} [Branching random walk]\\
		Suppose that Assumptions \ref{assumption_RW}  and \ref{as:BP} are satisfied. Then, the laws of $ \frac{M_n}{n^{1/r}}$ under $\P^*$ satisfy a large deviation principle 
		with rate function $I^{\textnormal{BRW}}$.
	\end{theorem}
\begin{remark}\label{Hvardes}	
In the proof of Theorem \ref{theorem_BRW}, we will use the following variational form of $H$.
 For $x < \alpha$ let
\begin{equation} \label{gammadef}
\gamma =\gamma(x) = 
		\left\{ \begin{array}{cc} 1 & \mbox{if } x\leq 0\\
		1- \bigl(\frac{x}{\alpha}\bigr)^r  & \mbox{if } x > 0.
				\end{array} \right.
	\end{equation}
Then	
		\begin{equation} \label{def_Hvar}
		H(x) = 
		\left\{ \begin{array}{cc} k^*\lambda_- (\alpha-x)^r & \mbox{if } p(0)+p(1)=0,\\
		\inf_{t \in [0,\gamma]} \Bigl\{ t \rho + \lambda_- \bigl((1-t)^{1/r}\alpha-x \bigr)^r  \Bigr\} & \mbox{if } p(0)+p(1)>0.
				\end{array} \right.
	\end{equation}	
One derives \eqref{def_H} from \eqref{def_Hvar} by a straightforward calculation.
\end{remark}

\begin{remark}\label{moregeneral}
For the upper large deviations, \eqref{lowtail} is not needed: if \eqref{upptail} holds and if
$\E[ X ] =0$ and $\E[ X^2 ] < \infty$,
then for $x \geq \alpha$, 
\begin{equation}
\lim_{n \to \infty} \P^* \Bigl(\frac{M_n}{n^{1/r}} \geq x \Bigr)= - \left(\lambda_+ x^r - \log m\right) \, .
\end{equation}
If \eqref{lowtail} is not satisfied, the asymptotics of
$\P^* \Bigl(\frac{M_n}{n^{1/r}} \leq x \Bigr)$
for $x < \alpha$ may change but our techniques of proof still apply.
\end{remark}
We now state the results for independent walks and then give some intuition for the rate functions.	
	Define the rate function for the maximum of 
	independent random walks in the Schröder case as
	\begin{equation} \label{def_I_ind_S}
		I^{\text{ind}}(x)=
		\begin{cases}
		\lambda_+ x^r - \log m & \text{for } x \geq \alpha,\\
		\rho \bigl(1 - \tfrac{\lambda_+ x^r}{\log m} \bigr)  &  \text{for } 0 \leq x \leq \alpha,\\
		k^* \lambda_- (-x)^r + \rho & \text{for } x \leq 0.
		\end{cases}
	\end{equation}
	Note that $\rho \bigl(1 - \tfrac{\lambda_+ x^r}{\log m} \bigr)= I^{\text{GW}}(I(x))$ for $ 0 \leq x < \alpha$, see \eqref{def_I} and \eqref{def_I_GW_S}.

	\begin{theorem}\label{theorem_ind_RW_Schröder}[Independent random walks, Schröder case]\\
		Suppose that Assumptions \ref{assumption_RW} and \ref{as:BP} are satisfied and that 
		$p(0)+p(1)>0$. Then, the laws of $\frac{\tilde{M}_n}{n^{1/r}}$ under $\P^*$ satisfy a large deviation principle with rate function $I^{\text{ind}}$.
	\end{theorem}
We now turn to the deviations of $\tilde{M}_n$ in the Böttcher case.
In the Böttcher case ($p(0)+p(1)=0$) we have $\rho=\infty$ and therefore $I^{\text{ind}}(x)=\infty$ for all $x<\alpha$. Hence, in this case the lower deviation probabilities 
	$\P^*( \tilde{M}_n \leq xn)$ for $x<\alpha$ decay faster than exponentially in $n$.
	
\begin{theorem}\label{theorem_ind_RW_B} [Independent random walks, Böttcher case]\\
	Suppose that Assumptions \ref{assumption_RW} and \ref{as:BP} are satisfied and that $p(0)+p(1)=0$. Then,
	\begin{align*}
		\log \P(\tilde{M}_n \geq xn^{1/r}) &\sim -(\lambda_+x^r-\log m)n & \mbox{for }x \geq \alpha, \\
		\log | \log \P(\tilde{M}_n \leq xn^{1/r})| &\sim - n \log k^* \left(1  - \frac{\lambda_+x^r}{\log m}\right) & \mbox{for } x \in [ 0, \alpha), \\
		\log \P(\tilde{M}_n \leq xn^{1/r}) &\sim -n(k^*)^n\lambda_- |x|^r& \mbox{for } x < 0.
	\end{align*}
\end{theorem}

Let us now give some intuition for the rate function $I^{\text{ind}}$ and describe the large deviation events leading to $\{\tilde{M}_n \geq xn^{1/r}\}$ and $\{M_n \geq xn^{1/r}\}$ for some $x>\alpha$, respectively 
	$\{\tilde{M}_n \leq xn^{1/r}\}$ and $\{M_n \leq xn^{1/r}\}$ for some $x<\alpha$.
	
For $x>\alpha$, the number of particles should be larger or equal than expected, i.e.~$Z_n \geq e^{nt}$ for some $t \geq \log m$. It turns out that $t=\log m$ is the optimal value: it is too expensive to produce $e^{nt}$ particles for some $t > \log m$.
If there are $e^{nt}$ particles at time $n$, the probability that at least one particle reaches $xn^{1/r}$ is of order $\exp(-I(x)n+tn+o(n))$ for 
	$t < I(x)$. Therefore,  the rate function is $I(x) - \log m$.
	This argument is the same for the independent walks and for the branching random walk.
Note that for $x > \alpha$ the rate of decay of $\log \P(\tilde{M}_n \geq xn^{1/r})$  in the Schröder case and the Böttcher case coincide, and concide with the rate function for branching random walk.\\
	Consider now independent walks in the Schröder case.
	If $0 \leq x< \alpha$, the probability that one particle reaches $xn^{1/r}$ is of order $\exp(-I(x)n+o(n))$. Hence, for every $\varepsilon>0$, if there are less than $e^{(I(x)- \varepsilon) n}$ 
	particles, the probability that none of these particles reaches $xn^{1/r}$ is close to 1. However, if there are more than $e^{(I(x)+ \varepsilon) n}$ particles, this probability decays 
	exponentially in $n$.
	If $x<0$, already the probability that a single particle is below $xn^{1/r}$ at time $n$ decays exponentially fast in $n$. Hence, if the number of particles $Z_n$ grows exponentially, the 
	probability that all particles are below $xn^{1/r}$ at time $n$ decays faster than exponentially. Therefore, the number of particles needs to grow subexponentially. Since $\rho$ does not 
	depend on the choice of $k$ in Lemma~\ref{lem:LDBP_S}, the optimal strategy for $\{\tilde{M}_n \leq xn^{1/r}\}$ is to  have only $k^*$ particles at time $n$ with all their positions being below $ xn^{1/r}$. 

For the branching random walk, in the Schröder case the strategy for 
	$x < \alpha$ goes as follows. At time $tn$ there are only $k^*$ particles, and 
	one of them has an offspring whose displacement is negative with a large absolute value, while all other $k^*-1$ particles have no offspring. Note that in the Schröder case either $k^*=1$ or $p(0) > 0$. Afterwards, all particles move and reproduce as usual.\\
Further notice, that in contrast to the case of independent random walks, the number of particles can also grow exponentially if $x<0$ in the Schröder case. It suffices to have a
small number of particles at time $tn$ for some $t \in [0,1]$.\\
To compare the rate functions in the Schröder case, note that the maximum of independent random walks stochastically dominates the maximum of the branching random walk, see 
Lemma~\ref{lemma_BRW_vs_ind_RW2}. Therefore, $I^{\text{ind}}(x) \leq I^{\text{BRW}}(x)$ for $x>\alpha$, respectively $I^{\text{ind}}(x) \geq I^{\text{BRW}}(x)$ for 
	$x<\alpha$. For $x<\alpha$, the inequality is in general strict. For $x>\alpha$, the rate functions coincide, see the argument above.

\section{Proofs}\label{sec:proofs}
\subsection{Auxiliary results}

	We often need to estimate the number of particles at time $n$, which has expectation $m^n$. Let $W_n = \frac{Z_n}{m^n}$ and $(\mathcal{F}_n)_{n \in \N}$ be the natural 
	filtration of the Galton-Watson process, i.e.~$\mathcal{F}_n= \sigma(Z_1, \ldots, Z_n)$. 
	Then $(W_n)_{n \in \N}$ forms a non-negative martingale with respect to the filtration $(\mathcal{F}_n)_{n \in \N}$. Therefore, $W_n \to W$ almost surely, where $W$ is an almost surely 
	finite random variable. The following well-known theorem shows that under our assumptions, the limit $W$ is non-trivial, i.e.~$\P^*(W>0)=1$. 
	\begin{theorem}[Kesten-Stigum Theorem] \label{Lemma_GWP_martingale}
		If Assumption \ref{as:BP} is satisfied, we have 
		\begin{equation*}
			\E[W]=1 \quad \text{and } \quad \P(W=0)=q < 1.
		\end{equation*}
	\end{theorem}
	A proof can e.g.~be found in \cite[Chapter 1, Section 10, Theorem 1]{AN04}.\\
	In the proof of Theorem~\ref{theorem_BRW} we use the fact that the maximum of independent random walks stochastically dominates the maximum of the branching random walk, see 
	Lemma~\ref{lemma_BRW_vs_ind_RW1} below. Lemma~\ref{lemma_BRW_vs_ind_RW2} comes from the following general statement which we will use in our proof as well.  Both lemmas are proven in~\cite[Lemma~5.1, 5.2]{GH18}.

\begin{lemma}  
\label{lemma_BRW_vs_ind_RW1}
		Let $(X_i)_{i \in \N}$ and $(Y_i)_{i \in \N}$ be independent sequences of (not necessarily independent) random variables. Furthermore, assume that the random variables 
		$Y_i, i \in \N$, have the same distribution. Then we have for all $k \in \N$ and $x \in \R$
		\begin{equation*}
		\P \Bigl( \max_{i \in \{1, \ldots, k\}} \{X_i+Y_1\} \leq x \Bigr) \geq \P \Bigl( \max_{i \in \{1, \ldots, k\}} \{X_i+Y_i\} \leq x \Bigr).
		\end{equation*}
	\end{lemma}
\begin{lemma} 
\label{lemma_BRW_vs_ind_RW2}
		For all $n \in \N$ and $x \in \R$ 
		\begin{equation}\label{domino}
			\P (M_n \leq x) \geq \P (\tilde{M}_n \leq x). 
		\end{equation}
	\end{lemma}

	The statement of Lemma~\ref{lemma_BRW_vs_ind_RW2} is also true with respect to $\P^*$.

We will also use the following tail asymptotic for stretched exponential random variables, which follows from~\cite[Theorem 8.3]{DDS08} .
\begin{lemma}
Let Assumption~\ref{assumption_RW} be in force.
		If $x_n \gg n^{\sigma}$, where $ \sigma = \frac{1}{2-2r}$, then
		\begin{equation}\label{sumasmax}
			\P\left[ S_n > x_n \right] \sim n \P\left[ X>x_n \right].
		\end{equation}
	\end{lemma}
	
\subsection{Main steps}  
	In this section we prove the main results of the paper.  We first prove the results for the branching random walk using the results for the independent walks. Then, we 
	give the proof of the results for the independent walks.

\subsection*{Branching random walk} \label{proofs_BRW}

\begin{proof}[Proof of Theorem \ref{theorem_BRW}]
\textbf{Case (1)}: $x > \alpha$\\
The upper bound immediately follows from Theorems~\ref{theorem_ind_RW_Schröder} and~\ref{theorem_ind_RW_B} combined with Lemma~\ref{lemma_BRW_vs_ind_RW2}. 
	It remains to prove the lower bound. For $v \in D_{n-1}$ denote the rightmost descendant of $v$ at time $n$ by $M^v_1$. By Lemma~\ref{lemma_BRW_vs_ind_RW1},
	\begin{align}
	\P^* \Bigl(\frac{M_n}{n^{1/r}} \geq x \Bigr)
	& = \P^* \Bigl( \max \limits_{v \in D_{n-1}} \frac{M_1^v - S_v}{n^{1/r}} + \frac{S_v}{n^{1/r}} \geq x \Bigr) \notag \\ 
	& \geq \P^* \Bigl( \max \limits_{v \in D_{n-1}} \frac{M_1^v - S_v}{n^{1/r}} + \frac{S_{n-1}}{n^{1/r}} \geq x \Bigr) \notag \\
	& \geq \P^* \Bigl(\max \limits_{v \in D_{n-1}} \frac{M_1^v-S_v }{n^{1/r}} \geq x \Bigr) \cdot \P \Bigl( \frac{S_{n-1}}{n^{1/r}} \geq 0 \Bigr). \label{proof_theorem_BRW_case_x>x*_1}
	\end{align}
	The second probability on the right hand side of \eqref{proof_theorem_BRW_case_x>x*_1} converges to $\frac{1}{2}$ by the central limit theorem. It remains to estimate the 
	first probability on the right hand side of \eqref{proof_theorem_BRW_case_x>x*_1}. Note that $(M_1^v - S_v)_{v \in D_{n-1}}$ are independent under $\P$. 
	Therefore, using the inequality $1-(1-z)^y \geq zy(1-zy)$ valid for $z \in [0,1]$ and $y\geq 0$, we get for $c > 0$
	\begin{align} \label{proof_theorem_BRW_case_x>x*_2}
	& \quad \ \P^* \Bigl(\max \limits_{v \in D_{n-1}} \frac{M_1^v - S_v}{n^{1/r}} \geq x \Bigr) \notag \\
	& \geq \P^* \bigl(Z_{n-1} \geq c m^n \bigr) \cdot \biggl( 1- \Bigl(1- \P \Bigl(\frac{X}{n^{1/r}} \geq x \Bigr) \Bigr)^{c m^n} \biggr) \notag\\
	& \geq \P^* \bigl(W_{n-1} \geq cm \bigr)  \P \Bigl(\frac{X}{n^{1/r}} \geq x \Bigr) c m^n \Bigl(1-\P \Bigl(\frac{X}{n^{1/r}} \geq x \Bigr) c m^n \Bigr).
	\end{align}
	For the first factor on the right hand side of \eqref{proof_theorem_BRW_case_x>x*_2} it holds that $\liminf_{n \to \infty} \P^*(W_{ n-1} \geq cm ) \geq \P^*(W > cm )>0$ for $c$ small enough.
	Furthermore, the last term on the right hand side of \eqref{proof_theorem_BRW_case_x>x*_2} converges to 1 by Assumption~\ref{assumption_RW}. 
	Combining \eqref{proof_theorem_BRW_case_x>x*_1} and \eqref{proof_theorem_BRW_case_x>x*_2} yields the lower bound.\\
\textbf{Case (2)}: $x < \alpha$ and $p(0)+p(1)>0$\\
Following the strategy explained in Section~\ref{sec:main}, we consider the situation where there are only $k^*$ particles at time $tn$ and one of them has an offspring with a large negative displacement while all others have no offspring. Afterwards, all particles move and reproduce as usual. For the lower bound, recall 
\eqref{gammadef}, let $t \in [0, \gamma]$ and fix 
	 $\varepsilon >0$. 
We have
	\begin{align} \label{proof_theorem_BRW_case_x<x*_1}
		\P^* \Bigl(\frac{M_n}{n^{1/r}} \leq x \Bigr)
		& \geq \P^* \Bigl(\frac{M_n}{n^{1/r}} \leq x \bigm\vert Z_{tn}=k^* \Bigr) \cdot \P^* (Z_{tn}=k^* ) \nonumber\\
		& \geq p(0)^{k^*-1} \P^* \Bigl(\frac{S_{tn}+M_{(1-t)n}}{n^{1/r}} \leq x \Bigr) \cdot \P^* (Z_{tn}=k^* ) \nonumber \\
		& \geq p(0)^{k^*-1} \P^* \Bigl(\frac{M_{(1-t)n}}{((1-t) n)^{1/r}} \leq \alpha + \varepsilon \Bigr) \cdot \P\Bigl(\frac{X_1}{n^{1/r}} \leq -\bigl( (\alpha+ \varepsilon)(1-t)^{1/r} -x \bigr) \Bigr) \nonumber \\
		& \quad \cdot \P(S_{tn-1} \leq 0) \cdot \P^* (Z_{tn}=k^* ).
	\end{align}
Since the first probability on the right hand side of \eqref{proof_theorem_BRW_case_x<x*_1} converges to 1 almost surely as $n \to \infty$ by 
\eqref{linspeed}, we get
\begin{align*}
\P^* \Bigl(\frac{M_n}{n^{1/r}} \leq x \Bigr) \geq \exp \Bigl( -\Bigl[t \rho + \lambda_- \Bigl((\alpha+ \varepsilon)(1-t)^{1/r}-x \Bigr)^r \Bigr]n + o(n) \Bigr).
\end{align*}
Since this inequality holds for all $t \in [0, \gamma]$, letting $n \to \infty$ followed by $\varepsilon \to 0$ yields, 
\begin{align*}
\liminf_{n \to \infty} \frac{1}{n} \log \P^* \Bigl(\frac{M_n}{n^{1/r}} \leq x \Bigr) \geq \sup_{t \in [0, \gamma ]}\left( - H(x)\right) = -\inf_{t \in [0, \gamma]}  H(x).
\end{align*}
In what follows, we often omit integer parts for better readability. For the upper bound, recalling \eqref{gammadef}, assume $t < \gamma$ and define
\begin{equation*}
T_n = \inf \Bigl\{t \geq 0 \colon Z_{tn} \geq n^3 \Bigr\}
\end{equation*}
and for $\varepsilon_1>0$ introduce the set
\begin{equation*} 
F=F(\varepsilon_1)= \Bigl\{\varepsilon_1, 2 \varepsilon_1, \ldots, \Bigl\lceil 
\gamma \varepsilon_1^{-1} \Bigr\rceil \varepsilon_1 \Bigr\}.
\end{equation*}
By the definition of $T_n$ we then have
\begin{align} \label{proof_theorem_BRW_case_x<x*_2}
& \quad \ \P^* \Bigl(\frac{M_n}{n^{1/r}} \leq x \Bigr) \nonumber\\
& \leq \P^* (T_n > \gamma ) + \sum_{t \in F} \P^* \Bigl(\frac{M_n}{n^{1/r}} \leq x \bigm\vert T_n \in \bigl(t - \varepsilon_1, t \bigr] \Bigr) \P^* \bigl(T_n \in \bigl(t - \varepsilon_1, t \bigr] \bigr) \nonumber \\
& \leq \P^* \bigl(  Z_{\gamma n} \leq n^3 \bigr)+ \sum_{t \in F} \P^* \Bigl(\frac{M_n}{n^{1/r}} \leq x \bigm\vert T_n \in \bigl(t - \varepsilon_1, t \bigr] \Bigr) \P^* (  Z_{(t- \varepsilon_1) n} \leq n^3).
\end{align}

Let $\varepsilon_2>0$ be such that $(1-t)^{1/r}(\alpha -\varepsilon_2) \geq x$ (which is possible since $t < \gamma$).
 Using Lemma~\ref{lemma_BRW_vs_ind_RW1} again,
\begin{align} \label{proof_theorem_BRW_case_x<x*_3}
& \quad \ \P^* \Bigl(\frac{M_n}{n^{1/r}} \leq x \bigm\vert T_n \in \bigl(t - \varepsilon_1, t \bigr] \Bigr) \nonumber \\
& \leq  \P^* \biggl(\max_{v \in D_{tn}} \frac{S_{tn}+ M^v_{(1-t)n}}{n^{1/r}} \leq x \Bigm\vert T_n \in \bigl(t - \varepsilon_1, t \bigr] \biggr) \nonumber \\
& \leq \P \Bigl(\frac{S_{tn}}{n^{1/r}} \leq -\bigl((1-t)^{1/r}(\alpha- \varepsilon_2)-x \bigr) \Bigr)+ \P^* \Bigl(\frac{M_{(1-t)n}}{((1-t)n)^{1/r}} \leq \alpha - \varepsilon_2 \Bigr)^{n^2} \notag \\
& \quad \ + \P^* \bigl( Z_{tn} \leq n^2 \bigm\vert  T_n \in (t - \varepsilon_1, t] \bigr).
\end{align}
Note that we ignored all but one particle at time $tn$. Further note that the second probability on the right hand side of \eqref{proof_theorem_BRW_case_x<x*_3} converges to $0$ by \eqref{linspeed} and therefore the second term in \eqref{proof_theorem_BRW_case_x<x*_3} decays faster than exponentially in $n$. For the third term on the right hand side of \eqref{proof_theorem_BRW_case_x<x*_3},
\begin{align} \label{proof_theorem_BRW_case_x<x*_4}
\P^* \bigl( Z_{tn} \leq n^2 \bigm\vert  T_n \in (t - \varepsilon_1, t] \bigr)
& \leq \P^* \bigl(\exists k \in \N \colon Z_k \leq n^2 \bigm\vert Z_0=n^3 \bigr) \notag\\
& \leq \binom{n^3}{n^2} q^{n^3-n^2} \leq \exp \bigl( (n^3-n^2) \log q + 3n^2 \log n \bigr).
\end{align}
In the second inequality we used the fact that for the event we consider, at most $n^2$ of the initial $n^3$ Galton-Watson trees can survive. Note that the Galton-Watson trees coming from different initial particles are independent.

Combining \eqref{proof_theorem_BRW_case_x<x*_2}, \eqref{proof_theorem_BRW_case_x<x*_3} and \eqref{proof_theorem_BRW_case_x<x*_4} and letting $\varepsilon_1, \varepsilon_2 \to 0$, we conclude after a straightforward calculation
\begin{equation*}
\limsup_{n \to \infty} \frac{1}{n} \log \P^* \Bigl(\frac{M_n}{n^{1/r}} \leq x \Bigr) \leq
 - \inf_{t \in [0,\gamma]} \Bigr\{t \rho + \lambda_- \bigl( \alpha(1-t)^{1/r} -x \bigr)^r \Bigr) \Bigr\} = -H(x).
\end{equation*}
\textbf{Case (3)}: $x < \alpha$ and $p(0)+p(1)=0$\\
Let $0< \varepsilon < \alpha - x$. Then we have
		\begin{align} \label{proof_theorem_BRW_case_x<x*_5}
			\P^* \Bigl(\frac{M_n}{n^{1/r}} \leq x \Bigr)
			& \geq \P^* \Bigl(\frac{M_n}{n^{1/r}} \leq x \bigm\vert Z_1=k^* \Bigr) \cdot \P^* (Z_1=k^* ) \nonumber\\
			& \geq \P^* \Bigl(\frac{X_1+M_{n-1}}{n^{1/r}} \leq x \Bigr)^{k^*} \cdot \P^* (Z_{1}=k^* ) \nonumber \\
			& \geq \P^* \Bigl(\frac{M_{n-1}}{ n^{1/r}} \leq \alpha + \varepsilon \Bigr)^{k^*} \cdot \P\Bigl(\frac{X_1}{n^{1/r}} \leq -\bigl( \alpha+ \varepsilon -x \bigr) \Bigr)^{k^*} 
			\cdot \P^* (Z_{1}=k^* ).
		\end{align}
		Since the first probability on the right hand side of \eqref{proof_theorem_BRW_case_x<x*_5} converges to $1$ as $n \to \infty$ by \eqref{linspeed}, we get
		\begin{align*}
			\P^* \Bigl(\frac{M_n}{n^{1/r}} \leq x \Bigr) \geq p(k^*)\exp \Bigl( -  \lambda_-k^* (\alpha+ \varepsilon-x)^r n + o(1) \Bigr).
		\end{align*}
		Letting $\varepsilon \to 0$ finishes the proof.
				For the upper bound note that $4 \log k^* \geq 2$ and in particular $|D_{4 \log n}| \geq n^2$ a.s. Take $0< \varepsilon < \alpha-x$. Using Lemma~\ref{lemma_BRW_vs_ind_RW1} and the inequality $(a+b)^k \leq 2^{k-1}(a^k + b^k)$ valid for all $a,b>0$ and $k \in \N$, we get
		\begin{align*}
			\P^* \Bigl(\frac{M_n}{n^{1/r}} \leq x \Bigr)  & \leq \P^* \Bigl(\frac{M_{n-1}+X}{n^{1/r}} \leq x \Bigr)^{k^*} \leq 
				\P^* \Bigl(\max_{v \in D_{4 \log n}} \frac{M^v_{n-1-4 \log n } + S_{4 \log n} + X}{n^{1/r}} \leq x \Bigr)^{k^*}  \\
& \leq \biggl(\P \Bigl(\frac{S_{4 \log n }}{n^{1/r}} \leq -(\alpha- \varepsilon-x ) \Bigr)+ \P^* \Bigl(\frac{M_{n-1- 4 \log n}}{n^{1/r}} \leq \alpha - \varepsilon \Bigr)^{n^2}	\biggr)^{k^*} \\			
				& \leq 2^{k^*-1}\P \Bigl(\frac{S_{4 \log n}}{n^{1/r}} \leq -(\alpha- \varepsilon-x ) \Bigr)^{k^*}+ 2^{k^*-1}\P^* \Bigl(\frac{M_{n-1- 4 \log n}}{n^{1/r}} \leq \alpha - \varepsilon \Bigr)^{k^*n^2}.   
	\end{align*}
Since the second probability on the right hand side  converges to $0$ as $n \to \infty$ by \eqref{linspeed}, we get, applying \eqref{sumasmax}
(with the random variables $-X_1, -X_2, \ldots $)
\begin{align*}
\P^* \Bigl(\frac{M_n}{n^{1/r}} \geq x \Bigr) \leq \exp \Bigl( - k^*\lambda_- (\alpha- \varepsilon-x)^r n + o(n) \Bigr).
\end{align*}
Letting $\varepsilon \to 0$ finishes the proof.
\end{proof}

\subsection*{Independent random walks} \label{proofs_ind_RW}

	\begin{proof}[Proof of Theorem \ref{theorem_ind_RW_Schröder}]
\textbf{Case (1)}: $x > \alpha$\\
For $z \in [0,1]$ and $y\geq 0$ we have $1-(1-z)^y \geq zy(1-zy)$. Following the strategy explained in Section~3, independence of the random walks and the aforementioned 
		inequality yield
		\begin{align} \label{proof_thm_ind_RW_1}
			\P^* \Bigl( \frac{\tilde{M}_n}{n^{1/r}} \geq x \Bigr)
				& = \E^* \biggr[ 1- \Bigr(1- \P \Bigl( \frac{S_n}{n^{1/r}} \geq x \Bigr) \Bigr)^{Z_n} \biggr] \nonumber\\
				& \geq \P^* \Bigl(Z_n \geq \frac{1}{2} m^n \Bigr) \cdot  \biggr( 1- \Bigr(1- \P \Bigl( \frac{S_n}{n^{1/r}} \geq x \Bigr) \Bigr)^{\frac{1}{2} m^n} \biggr) \nonumber \\
				& \geq \P^* \Bigl(W_n \geq \frac{1}{2} \Bigr) \P \Bigl( \frac{S_n}{n^{1/r}} \geq x \Bigr) \frac{1}{2} m^n \Bigl( 1- \P \Bigl( \frac{S_n}{n^{1/r}} \geq x \Bigr) \frac{1}{2} m^n \Bigr).
		\end{align}
		By Lemma~\ref{lem:LDP_RW}, $\P \bigl( \frac{S_n}{n^{1/r}} \geq x \bigr) \frac{1}{2} m^n \to  0$ as $n \to \infty$, since $\log m < I(x)$. 
		For the first factor on the right hand side of \eqref{proof_thm_ind_RW_1} we have 
		$\liminf_{n \to \infty} \P^* (W_n \geq \frac{1}{2} ) \geq \P^*(W > \frac{1}{2}) >0$, since $\E^*[W]=\frac{1}{1-q}>1$ by Theorem~\ref{Lemma_GWP_martingale}. This establishes the estimate
		\begin{equation*}
			\P^* \Bigl( \frac{\tilde{M}_n}{n^{1/r}} \geq x \Bigr) \geq  \exp \bigl( - (I(x) - \log m )n + O(1) \bigr) = \exp \bigl( - (\lambda_+ x^r - \log m )n + O(1) \bigr).
		\end{equation*}
		For the upper bound, the Markov inequality yields, recalling \eqref{defcopy}
		\begin{align*}
			\P^* \Bigl( \frac{\tilde{M}_n}{n^{1/r}} \geq x \Bigr) & = \P^* \Bigl( \sum_{i=1}^{Z_n} \mathds{1}_{\{S_n^{(i)} \geq n^{1/r}x\}} \geq 1 \Bigr) \leq \P \Bigl( \frac{S_n}{n^{1/r}} \geq x \Bigr)\E^*[Z_n] \\
				& \leq \P \Bigl( \frac{S_n}{n^{1/r}} \geq x \Bigr) \frac{m^n}{1-q},
		\end{align*}
		which immediately implies the claim by the merit of Lemma~\ref{lem:LDP_RW}\\
		
 \textbf{Case (2)}: $0 < x < \alpha$\\
Since the rate function $I$ is strictly increasing on the interval $[0,\alpha]$, we can choose $\varepsilon>0$ such that $\varepsilon < I(x) < \log m - \varepsilon$. We prove the 
	upper bound first. Using the inequality $1-y \leq e^{-y}$ and the third part of Lemma~\ref{lem:LDBP_S}, we have for $n$ large enough
	\begin{align*}
		\P^* \Bigl(\frac{\tilde{M}_n}{n^{1/r}} \leq x \Bigr)
		& = \E^* \biggl[\Bigl(1-\P \Bigl(\frac{S_n}{n^{1/r}} > x \Bigr)\Bigr)^{Z_n} \biggr] \leq \E^* \biggl[\exp \Bigl(-\P \Bigl(\frac{S_n}{n^{1/r}} > x \Bigr) Z_n \Bigr) \biggr]\\
		& \leq \P^* \Bigl(Z_n \leq e^{(I(x)+\varepsilon)n} \Bigr) + \exp \bigr(-e^{\varepsilon n +o(n)} \bigr)= \exp \Bigl(- \bigl( I^{\text{GW}}(I(x) +\varepsilon \bigr) n + o(n) \Bigr).
	\end{align*}
	Letting $\varepsilon \to 0$ gives the desired upper bound, since $I^{\text{GW}}$ defined in \eqref{def_I_GW_S} is continuous. 
	The proof for the lower estimate is similar. Since $\P (\frac{S_n}{n^{1/r}} > x) < e^{-1}$ for $n$ large enough,  inequality $1-y>e^{-ey}$ for $y \in [0, e^{-1}]$
	yields 
	\begin{align*}
		\P^* \Bigl(\frac{\tilde{M}_n}{n^{1/r}} \leq x \Bigr)
		& = \E^* \biggl[\Bigl(1-\P \Bigl(\frac{S_n}{n^{1/r}} > x \Bigr)\Bigr)^{Z_n} \biggr] \geq \E^* \biggl[\exp \Bigl(-e \cdot \P \Bigl(\frac{S_n}{n^{1/r}} > x \Bigr) Z_n \Bigr) \biggr]\\
		& \geq \P^* \Bigl(Z_n \leq e^{(I(x)-\varepsilon)n} \Bigr) \cdot \exp \bigr(-e^{-\varepsilon n +o(n)}  \bigr)= \exp \Bigl(- I^{\text{GW}}(I(x) - \varepsilon ) n + o(n) \Bigr).
	\end{align*}
	Letting $\varepsilon \to 0$ concludes the proof for $0<x<\alpha$.\\
	
\textbf{Case (3)}: $x \leq 0$\\
We first consider $x <0$. For the upper bound we have for $K \in \N$
	\begin{align}\label{eq:proof_thm_ind_RW_3}
		\P^* \Bigl(\frac{\tilde{M}_n}{n^{1/r}} \leq x \Bigr) &= \E^* \biggl[\P \Bigl(\frac{S_n}{n^{1/r}} \leq x \Bigr)^{Z_n} \biggr] \nonumber \\
		 	& \leq \sum_{k=1}^K \P \Bigl(\frac{S_n}{n^{1/r}} \leq x \Bigr)^k \P^*(Z_n=k) + \P \Bigl(\frac{S_n}{n^{1/r}} \leq x \Bigr)^K.
	\end{align}
	By the first part of Lemma~\ref{lem:LDBP_S}, the probability $\P(Z_n=k)$ is of order $\exp(-\rho n + o(n))$ for all $k \in A$ (defined in \eqref{def_A}) and $\P(Z_n=k)=0$ otherwise. 
	Therefore, for all $K \in \N$,  
	\begin{equation*}
		\limsup_{n \to \infty} \frac{1}{n} \log \P^* \Bigl(\frac{\tilde{M}_n}{n^{1/r}} \leq x \Bigr)  \leq  \max \bigl\{-(k^*I(x)+\rho), -KI(x) \bigr\}. 
	\end{equation*}
	Hence, letting $K \to \infty$ proves the upper bound since $I(x)>0$ for $x<0$.  The lower bound can be established similarly:
	\begin{align*}
		\P^* \Bigl(\frac{\tilde{M}_n}{n^{1/r}} \leq x \Bigr)
		& = \E^* \biggl[\P \Bigl(\frac{S_n}{n^{1/r}} \leq x \Bigr)^{Z_n} \biggr] \geq \P \Bigl(\frac{S_n}{n^{1/r}} \leq x \Bigr)^{k^*} \cdot \P(Z_n=k^*)\\
		&= \exp \bigl( -(k^*I(x)+\rho)n +o(n) \bigr),
	\end{align*}
	which shows the lower bound. For $x=0$ the result follows from continuity of the rate function $I^\textnormal{ind}$ at 0.
\end{proof}

\begin{proof}[Proof of Theorem~\ref{theorem_ind_RW_B}]
\textbf{Case (1)}: $x > \alpha$\\
	In this case one can use exactly the same arguments as in the proof of Theorem~\ref{theorem_ind_RW_Schröder}.\\
	
\textbf{Case (2)}: $0 \leq x < \alpha$\\
	Take $a \in (k^*, m)$ such that
	\begin{equation*}
		\log a = I(x) \frac{\log m - \log k^*}{\log m} + \log k^*.
	\end{equation*}
	Lemma~\ref{lem:LDBP_B} provides us with
	\begin{equation*}
		\log \left(- \log \P^*(Z_n\leq a^n)\right) \sim - n \log k^* \left(1  - \frac{\lambda_+x^r}{\log m}\right),
	\end{equation*}
	while the choice of $a$ guarantees
	\begin{equation*}
		\log \left(a^n \P(S_n\leq xn^{1/r})\right) \sim - n \log k^* \left(1  - \frac{\lambda_+x^r}{\log m}\right).
	\end{equation*}
	If we now recall the bounds established in the proof of Theorem~\ref{theorem_ind_RW_Schröder}, i.e.
	\begin{equation*}
		\P(Z_n\leq a^n) \exp ( -e a^n \P(S_n>x n^{1/r}) ) \leq \P^*(\tilde{M}_n \leq xn^{1/r})
	\end{equation*}
	and
	\begin{equation*}
		\P^*(\tilde{M}_n \leq xn^{1/r}) \leq \P(Z_n\leq a^n)  + \exp ( -a^n \P(S_n>x n^{1/r}) ) 
	\end{equation*}
	the claim will follow.\\
	
	\textbf{Case (3)}: $x < 0$\\
For the upper bound just write
	\begin{equation*}
		\P^*(\tilde{M}_n \leq xn^{1/r}) = \E^*[ \P(S_n\leq xn^{1/r})^{Z_n}] \leq \P(S_n\leq xn^{1/r})^{(k^*)^n}
	\end{equation*}
	which implies the claim. For the lower bound we have
	\begin{equation*}
		\P^*(\tilde{M}_n \leq xn^{1/r}) = \E^*[ \P(S_n\leq xn^{1/r})^{Z_n}] \geq \P(S_n\leq xn^{1/r})^{(k^*)^n} \P(Z_n=(k^*)^n)
	\end{equation*}
	since
	\begin{equation*}
		\log \P(Z_n=(k^*)^n) = \frac{(k^*)^n-1}{k^*-1} \log p(k^*)  = o(n(k^*)^n)
	\end{equation*}
	and the claim follows.
\end{proof}

\bibliographystyle{amsplain}

\Addresses

\end{document}